\documentclass[a4paper,oneside,12pt]{article}
\usepackage{amsmath,amsfonts,amscd,amssymb,amsthm}
\usepackage[english]{babel}
\usepackage{graphicx}
\usepackage{upgreek}
\usepackage{mathrsfs}
\usepackage{dsfont}
\usepackage[hang,small,labelfont=bf,up,textfont=it,up]{caption}
\usepackage{mathtools}
\usepackage{color}
\usepackage{tikz}
\usepackage{soul}
\usepackage{cancel}
\definecolor{Red}{cmyk}{0,1,1,0}

\definecolor{Blue}{cmyk}{1,1,0,0}

\newtheorem{theorem}{Theorem}
\newtheorem{definition}{Definition}

\newtheorem{corollary}{Corollary}

\newtheorem{proposition}{Proposition}

\newtheorem{remark}{Remark}
\newtheorem{theoremalpha}{Theorem}

\usepackage[pdfpagelabels=true,unicode=true]{hyperref}
\hypersetup
{
	pdfauthor={Leandro Chiarini and Leandro Cioletti},
	pdftitle={Bounds on the Quenched Pressure and  Main Eigenvalue of the Ruelle 
	Operator for Brownian Type Potentials},
	pdfkeywords={Brownian motion, One-dimensional lattice, Quenched Pressure, 
	Random potentials, Ruelle operator, Thermodynamic formalism},
	pdfsubject={The RPF Theorem for Brownian potentials},
	pdflang={en-US}
}

\title{Bounds on the Quenched Pressure and  Main Eigenvalue of the Ruelle 
Operator for Brownian Type Potentials}
\author{$(\mathrm{L.C.})^2$}
\date{\today}
\begin{document}
\maketitle
\begin{abstract}
In this paper we consider a random potential derived 
from the Brownian motion. We obtain upper and lower
bounds for the expected value of the main eigenvalue 
of the associated Ruelle operator
and for its quenched topological pressure. 
We also exhibit an isomorphism between 
the space $C(\Omega)$ endowed with its standard norm and 
a proper closed subspace of the Skorokhod space 
which is used to obtain a stochastic functional equation 
for the main eigenvalue and for its associated eigenfunction. 
\end{abstract}


\section{Introduction}\label{secao-introducao}

The basic idea of the Ruelle operator is based on the 
transfer matrix method introduced by Kramers and Wannier and 
independently by Montroll to explicitly compute partition functions 
in Statistical Mechanics. This operator was introduced in 1968 
by David Ruelle \cite{MR0234697} and was used to describe local 
to global properties of systems with infinite range interactions. 
Among other things, Ruelle used this operator to prove the existence of a unique 
DLR-Gibbs measure for a lattice gas system depending on infinitely many coordinates.
The acronym DLR stands for Dobrushin, Lanford and Ruelle. 

With the advent of the Markov partitions due to Y. Sinai, 
remarkable applications of this operator to Hyperbolic dynamical systems 
on compact manifolds were further obtained by Ruelle, Sinai and Bowen, 
see \cite{MR0234697,MR0399421,MR2423393}.
Since its creation, this operator remains to have a great influence in many fields
of pure and applied Mathematics. 
It is a powerful tool to study topological dynamics, 
invariant measures for Anosov flows, statistical mechanics in one-dimension,
meromorphy of the Zelberg and Ruelle dynamical zeta functions, multifractal analysis, 
conformal dynamics in one dimension and fractal dimensions of horseshoes, just to name
a few. For those topics, we recommend, respectively,  
\cite{MR3024807,MR2423393,MR556580,MR1143171,MR742227,MR1085356,MR1920859,MR0399421} 
and references therein. 

In order to describe our contributions,
we shall start by briefly recalling the classical setting of the 
Ruelle operator. A comprehensive exposition of this subject can be found 
in \cite{MR1793194,MR2423393,MR1085356}.

We now introduce the state space and the dynamics that will guide the following discussions. 
In this paper, the state space will be the Bernoulli space 
$\Omega = \{0,1\}^\mathbb{N}$ and the dynamics is given by 
$\sigma: \Omega \to \Omega$, the left shift map. 
When topological concept are concerned, the state space $\Omega$ 
is regarded as a metric space where
the distance between points $x$ and $y \in\Omega$ 
is given by
\[
    d(x,y)= \frac{1}{2^N}, 
    \text{ where } N = \inf \{ i \in \mathbb{N}: x_i \neq y_i\}.
\]
We shall remark that $\Omega$ is a compact metric space when endowed with this metric. 
All of the results presented in this paper can be naturally extended to 
$\{0,1, \dots, m-1\}^\mathbb{N}$, 
$m \geq 2$, but for the sake of simplicity, we 
restrain ourselves to the binary case. A summary of the changes required 
to study this more general case are presented in Section \ref{sec-remarks}.

Back to the case $\Omega = \{0,1\}^\mathbb{N}$,
the set of all real continuous functions defined over $\Omega$ will be 
denoted by $C(\Omega)$ and
the Ruelle transfer operator, or simply the Ruelle operator,
associated to a continuous function $f: \Omega \to \mathbb{R}$ 
will be denoted by $\mathscr{L}_{f}$. 
This operator acts on the space of all continuous functions 
sending $\varphi$ to $\mathscr{L}_{f}(\varphi)$
which is defined for any $x\in\Omega$ by the following expression
\[
\mathscr{L}_{f}(\varphi)(x) 
= 
\sum_{a=0,1} \exp(f(ax))\varphi(ax),
\] 
where $ax\equiv (a,x_1,x_2,\ldots)$. 
The function $f$ is usually called {\bf potential}.

In this work a major role will be played by 
$C^{\gamma}(\Omega)$, the space of $\gamma$-H\"older continuous functions, 
with the H\"older exponent $\gamma$ satisfying $0<\gamma<1/2$. This space is
the set of all functions $f:\Omega\to\mathbb{R}$ satisfying 
\[
\mathrm{Hol}(f)
=
\sup_{x,y\in\Omega:\, x\neq y}
\dfrac{|f(x)-f(y)|}{d(x,y)^{\gamma}}
<+\infty.
\] 
Equipped with its standard norm $\|\cdot\|_{\gamma}$ 
given by 
$
C^{\gamma}(\Omega)\ni f
\mapsto
\|f\|_{\gamma}\equiv \|f\|_{\infty}+ \mathrm{Hol}(f)
$, 
the space $C^{\gamma}(\Omega)$ is a Banach space, 
see \cite{MR1085356} for more details.

Furthermore, we present the suitable version of the Ruelle-Perron-Fr\"obenius Theorem that will be required for our discussions. See 
\cite{MR1793194,MR2423393,MR1085356} for a proof.

\begin{theoremalpha}[Ruelle-Perron-Fr\"obenius, RPF]\label{teo-rpf-holder-case}
If $f$ is a potential in $C^{\gamma}(\Omega)$ for some $0<\gamma<1$, 
then $\mathscr{L}_f: C^{\gamma}(\Omega) \to C^{\gamma}(\Omega)$
have a simple positive eigenvalue of maximal modulus $\lambda_f$
with a corresponding strictly positive eigenfunction $h_f$ 
and a unique Borel probability measure $\nu_{f}$ on $\Omega$ such that,
\begin{itemize}
\item[i)] 
the remainder of the spectrum of 
$\mathscr{L}_f: C^{\gamma}(\Omega) \to C^{\gamma}(\Omega)$
is contained in 
a disc of radius strictly smaller than $\lambda_f$;

\item[ii)] the probability measure $\nu_{f}$ satisfies 
$\mathscr{L}^*_{f}(\nu_{f})=\lambda_{f}\nu_{f}$, where 
$\mathscr{L}^*_{f}$ denotes the dual of the Ruelle operator; 
  
\item[iii)] 
for all continuous functions $\varphi\in C(\Omega)$ we have
\[
\lim_{n\to\infty}
\left\|
\frac{\mathscr{L}^{n}_{f}(\varphi)}{\lambda_{f}^{n}}
-h_f \int_{\Omega} \varphi\,  d\nu_{f}
\right\|_{\infty}
=
0.
\]
\end{itemize}
\end{theoremalpha}

Whereas in the classical theory of the Ruelle operator, $f$ is some fixed H\"older, 
Walters \cite{MR2342978,MR1841880} or Bowen \cite{MR2423393,MR1783787} potential, 
here we are interested in obtaining almost certain results for 
random potentials. 
\begin{definition}[Random Potentials]\label{def-random-potential}
Let $(\Psi,\mathcal{F},\mathbb{P})$ be a complete probability space.
A random potential is a map $X:\Psi\times\Omega\to\mathbb{R}$
satisfying for any fixed $x\in\Omega$ that $\psi\mapsto X(\psi,x)$
is $\mathcal{F}$-measurable.
A random potential $X$ is said to be a {\bf continuous random potential}
if $x\mapsto X(\psi,x)\in C(\Omega)$ for $\mathbb{P}$-almost all 
$\psi\in\Psi$.
\end{definition}

Although some of our results could be stated for a large 
class of random potentials, we focus on the 
random potentials of Brownian type
(see the next section for the definition of such potentials). 

When moving into the realm of random potentials the basic theory
of the Thermodynamic Formalism needs to be reconstructed from scratch.
In \cite{MR958289} random potentials were considered and this is the 
setting of our work. Latter, more general settings were introduced.
The theory was extended to cover countable random Markov subshifts
of finite type, allowing randomness to be considered also 
in the adjacency matrix, we refer the reader to 
\cite{MR1336700,MR1402413,MR2410952,MR958289,MR1739595,MR2399927,MR2186245,MR2866474},
for this more general setting. 

As mentioned above the randomness here is considered exclusively in 
the potential. Since this is simpler setting, 
for the reader's convenience  we proved  
here, in full details for the Brownian potential, that important quantities such 
as the spectral radius, the topological pressure 
and the main eigenvalue of the Ruelle operator can be associated to random variables 
with finite first moment in Wiener spaces. 
At the end of Section \ref{sec-brown-pot} we explain how to  
obtain similar results for Brownian potential 
using the general setting considered in \cite{MR2399927,MR2186245}.

Working with random potentials, we are 
led to consider questions similar in spirit to 
the ones we ask when studying disordered systems in 
Statistical Mechanics, for instance the existence of annealed and 
quenched free energy and pressure, respectively, 
see \cite{MR1026102}. 
Our approach provides a new way to obtain the existence and finiteness 
of quenched topological pressure defined by
\[
P^{\mathrm{quenched}}(X)
\equiv
\mathbb{E}\left[\sup_{\mu\in \mathscr{P}_{\sigma}(\Omega)}
\left\{h(\mu)+\int_{\Omega} X\, d\mu \right\}\right],
\]
(note that the existence of the rhs above is itself a deep question) as well as the identity 
$
P^{\mathrm{quenched}}(X)
=
\mathbb{E}\left[\lim_{n\to\infty} 
n^{-1}\log \mathscr{L}^n_{X}(1)(\sigma^n (x))\right],
\ \forall x\in\Omega
$
for a large class of continuous random potential $X$
not invoking the classical replica trick. 
On the other hand, despite the progress achieved with this method, 
it has the disadvantage of not being able
to exhibit explicitly these quantities. 

Our main results are Theorems 
\ref{teo-banacho-iso}, \ref{prop-estimativa-valor-esperado-lambda},
the expression \eqref{eq-expressao-lambdaB}, and 
the bounds for the quenched pressure in Section \ref{sec-quenched-pressure}.

The paper
is organized as follows. 
In Section \ref{sec-brown-pot} we introduce the so-called Brownian potential and 
we prove that the main eigenvalue for the random Ruelle operator associated to the Brownian potential almost surely exists. 
We also prove in this section that phase transition for this random potential, in the sense 
of the Definition \ref{def-transicao-fase}, is almost surely absent. 
In Section \ref{sec-D-sim-C} and \ref{sec-fun-eq} we prove a Banach isomorphism between $C(\Omega)$
and a certain closed subspace of the Skorokhod space,
denoted by $\mathscr{D}[0,1]$.
With the help of this isomorphism we obtain a stochastic functional
equation satisfied by both the main eigenvalue and its associated
eigenfunction. Then a representation of the main eigenvalue in terms
of the Brownian motion and a limit of a quotient between 
linear combinations of log-normal random variables are also obtained. 
Section \ref{sec-bounds-lambda} is devoted to obtain upper and lower bounds 
for the expected value of the main eigenvalue and 
the main idea is the use of the reflection principle 
of the Brownian motion to overcome the complex 
combinatorial problem arising from the representation obtained
in Section \ref{sec-fun-eq}. In Section \ref{sec-quenched-pressure} we use the results of Sections \ref{sec-brown-pot}-\ref{sec-bounds-lambda} 
to give a proof of existence and finiteness of the quenched
topological pressure. In Section 7 we briefly discuss some natural  
extensions of our results, as to more general alphabets and to
other random potentials. 
We also explain the difficulties in improving our estimates and 
obtaining the existence of the annealed pressure.

\section{Brownian Type Potentials} \label{sec-brown-pot}

Let $\{B_t: t\in [0,1]\}$ be the standard one-dimensional Brownian motion,
i.e., a stochastic process with continuous 
paths almost surely, stationary independent increments and 
$B_t\sim \mathcal{N}(0,t)$ for each $t\geq 0$, see \cite{MR1121940,MR2574430}.

The idea is to use the almost certain continuity of its paths  
$t\mapsto B_t$ to define an almost certain continuous 
random potential on the symbolic space $\Omega=\{0,1\}^{\mathbb{N}}$.

The construction of this random potential is as follows.
We consider the function (abusing notation) $t:\Omega\to [0,1]$ 
from the symbolic space to the closed unit interval $[0,1]$, 
defined for each $x\in\Omega$ 
by the expression $t(x)= \sum_{n\geq 1}2^{-n}x_n$.
The {\bf Brownian potential} is the 
random mapping from $\Omega$ to $\mathbb{R}$
given by
\[
\Omega\ni x\mapsto B_{t(x)}.
\] 

\begin{figure}[h!]
\centering
\includegraphics[width=0.5\linewidth]{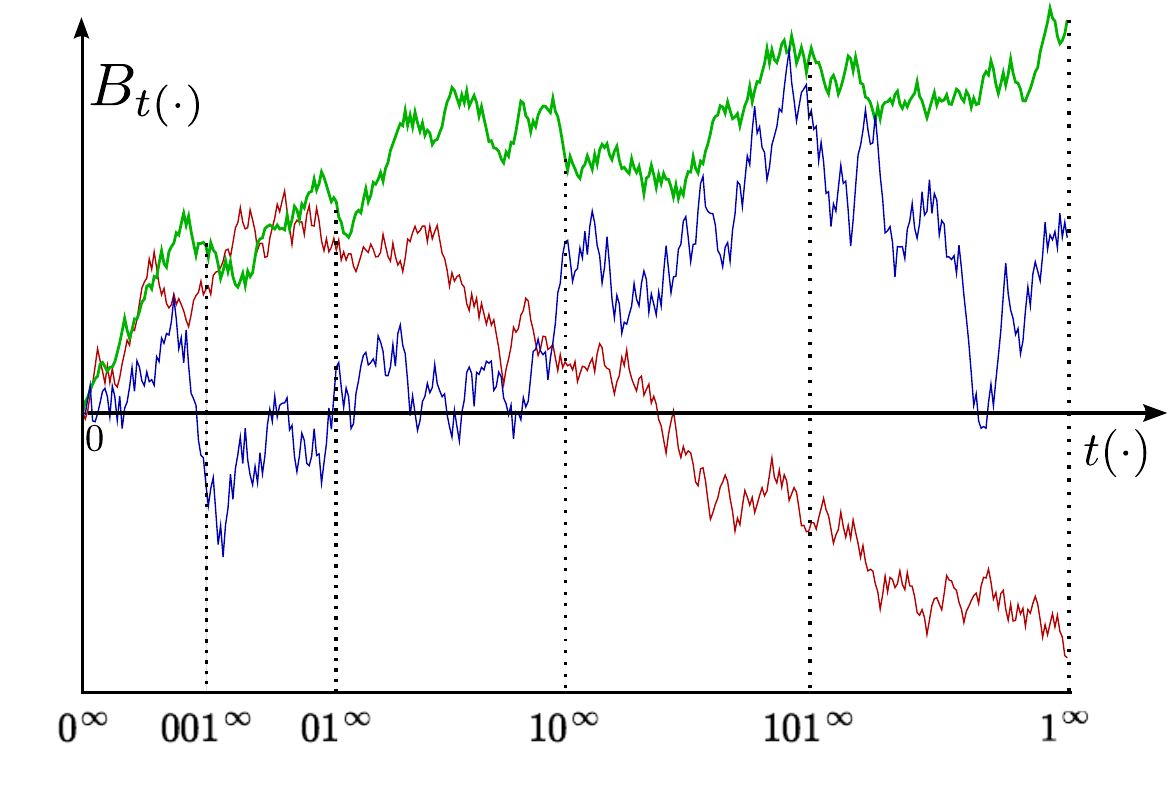}
\caption[Samples]{Samples of the Brownian potential. \\
$a_1\dots a_n^\infty := (a_1, \dots, a_n, a_n, a_n, \dots)$}
\label{fig:sample-path-brownian-motion}
\end{figure}

From the very definition of the Brownian motion follows that the
Brownian potential is almost certain a continuous potential
over the symbolic space $\Omega$. Therefore we have almost certain a well 
defined Ruelle operator 
$\mathscr{L}_{B_{t(\cdot)}}$ sending $C(\Omega)$ to itself.
This ``random" operator takes a continuous function $\varphi$ to
another continuous function $\mathscr{L}_{B_{t(\cdot)}}(\varphi)$ 
that is given for each $x\in\Omega$ by 
\begin{align}\label{def-oper-ruelle-brown}
\mathscr{L}_{B_{t(\cdot)}}(\varphi)(x)
=
\sum_{y\in\Omega: \sigma(y)=x}
\exp(B_{t(y)})\varphi(y).
\end{align}

We observe that rhs of \eqref{def-oper-ruelle-brown} let it clear that 
$\mathscr{L}_{B_{t(\cdot)}}(\varphi)(x)$ is a random 
variable, for any choice of $x\in\Omega$ and $\varphi\in C(\Omega)$. 

\begin{proposition}\label{prop-lambda-B-random-variable}
The spectral radius $\lambda_{B}$ of the random Ruelle operator
$\mathscr{L}_{B_{t(\cdot)}}$ is a non-negative random variable.
\end{proposition}
\begin{proof}
Of course, the only issue here is about measurability
and we show how to overcome this by using the Gelfand Formula and 
expression \eqref{def-oper-ruelle-brown}.
Indeed, by applying iteratively the expression 
\eqref{def-oper-ruelle-brown} one can see that 
$\mathscr{L}_{B_{t(\cdot)}}^{n}(\varphi)(x)$ is a random variable
for any choice of $\varphi\in C(\Omega)$ and $x\in\Omega$.
For any fixed $n\in\mathbb{N}$ and $\varphi\in C(\Omega)$ 
there exists, by the definition of the supremum, a sequence
$(x_k)$ in $\Omega$ such that 
$\| \mathscr{L}_{B_{t(\cdot)}}^{n}(\varphi) \|_{\infty}
=
\lim_{k\to\infty} |\mathscr{L}_{B_{t(\cdot)}}^{n}(\varphi)(x_k)|$
so proving that 
$\| \mathscr{L}_{B_{t(\cdot)}}^{n}(\varphi) \|_{\infty}$ is a random 
variable.  Since the
Gelfand's formula ensures that the spectral radius 
of $\mathscr{L}_{B_{t(\cdot)}}$ is given by 
\[
\lambda_{B}
=
\lim_{n\to\infty}
(\sup\{ 
\| \mathscr{L}_{B_{t(\cdot)}}^{n}(\varphi) \|_{\infty}
:
\varphi\in C(\Omega)\ \text{and}\ \|\varphi\|_{\infty}=1
\}
)^{\frac{1}{n}}
\]
we can argue similarly as above to conclude that 
$\lambda_B$ is a random variable.
\end{proof}

\begin{proposition}\label{prop-auto-val-func-sao-rv} The random variable $\lambda_{B}$ 
is the main eigenvalue of the random Ruelle operator 
$\mathscr{L}_{B_{t(\cdot)}}$ almost surely. Moreover this operator has a 
positive continuous eigenfunction $h_{B}$ and, additionally,
$h_{B}(x)$ is a random variable, for each $x\in\Omega$.
\end{proposition}

\begin{proof}
A typical trajectory of the Brownian motion 
in the unit interval $[0,1]$ is $\gamma$-H\"older continuous 
function for any $\gamma<1/2$, see \cite{MR1121940,MR2574430}.
A simple computation shows that 
$|t(x)-t(y)|\leq d(x,y)$, where $d(x,y)=2^{-N}$ and
$N = \inf\{n: x_j=y_j\ \text{for}\ 1\leq j< n\ \text{and}\ x_n\neq y_n \}$.
Therefore almost certain we have 
\[
|B_{t(x)}-B_{t(y)}|
\leq 
\mathrm{const.}|t(x)-t(y)|^{\gamma}
\leq 
\mathrm{const.}
(d(x,y))^{\gamma}.
\]
So our random potentials $x\mapsto B_{t(x)}$
are almost certain $\gamma$-H\"older continuous for any
$\gamma<1/2$.
From the classical Ruelle-Perron-Fr\"obenius
Theorem follows that $\lambda_{B}$, 
the random spectral radius of $\mathscr{L}_{B_{t(\cdot)}}$, 
is actually the main eigenvalue and it has geometric multiplicity 
one almost surely. From item iii) of Theorem \ref{teo-rpf-holder-case}
we have for each $x\in\Omega$ that  
\[
\lambda_{B}^{-n}\cdot \mathscr{L}^{n}_{B_{t(\cdot)} }(1)(x) 
\to
h(x).
\]
Since the lhs above is a random variable it follows that
$h(x)$ is also a random variable.
\end{proof}

Before proceed we recall some elementary facts about dual 
of the Ruelle operator associated to a deterministic 
general continuous potential. 
As we mentioned before $(\Omega,d)$ is a compact metric space
so one can apply the Riesz-Markov Theorem to prove that $C^*(\Omega)$
is isomorphic to $\mathscr{M}_s(\Omega)$, the space of all signed Radon measures.
Therefore we can define $\mathscr{L}^*_{f}$, the dual of the Ruelle operator,
as the unique continuous map from $\mathscr{M}_s(\Omega)$
to itself satisfying for each $\gamma\in\mathscr{M}_s(\Omega)$
the following identity
\begin{align}\label{eq-dualidade}
\int_{\Omega} \mathscr{L}_{f}(\varphi)\, d\gamma
=
\int_{\Omega} \varphi\, d(\mathscr{L}^*_{f}\gamma)
\qquad \forall \varphi\in C(\Omega).
\end{align}

From the positivity of $\mathscr{L}_f$ follows that
the map
$
\gamma
\mapsto
\mathscr{L}^*_{f}(\gamma)/
\mathscr{L}^*_{f}(\gamma)(1)
$
sends the space of all Borel probability measures
$\mathscr{P}(\Omega)$ to itself.
Since $\mathscr{P}(\Omega)$ is convex set and
compact in the weak topology (which is Hausdorff in this case)
and the mapping
$
\gamma
\mapsto
\mathscr{L}^*_{f}(\gamma)/
\mathscr{L}^*_{f}(\gamma)(1)
$
is continuous, the Tychonoff-Schauder Theorem ensures the existence of
at least one Borel probability measure $\nu_f$ such that
$
\mathscr{L}^*_{f}(\nu_f) = \mathscr{L}^*_{f}(\nu_{f})(1) \nu_f.
$
The eigenvalue
$\lambda_f\equiv \mathscr{L}^*_{f}(\nu_f)(1)$ is positive and
equals to the spectral radius of $\mathscr{L}_{f}$ and therefore
independent of the choice of the fixed point $\nu_f$, see
\cite{MR1793194,MR2423393,MR1085356}. 

If $X$ is a continuous random potential, then 
the Ruelle operator $\mathscr{L}_{X}$ sends
 $C(\Omega)$ to itself almost surely. We claim that for any
continuous random potential $X$ the spectral radius 
of the operator $\mathscr{L}_{X}$, denoted by $\lambda_{X}$
is a random variable. In fact, the claim follows
from the arguments given in  
the proof of Proposition \ref{prop-lambda-B-random-variable} 
replacing $B_{t(\cdot)}$ by $X$. 

\begin{definition}[Phase Transition]\label{def-transicao-fase}
Let $X$ be a continuous random potential and 
$\lambda_X$ the random spectral radius of $\mathscr{L}_{X}$. 
We say that the random potential $X$ does not present phase 
transition if the eigenvalue problem $\mathscr{L}_{X}^{*}\nu = \lambda_{X}\nu$ 
has  a unique solution in $\mathscr{P}(\Omega)$ almost surely.
Otherwise, we say that $X$ presents phase transition.
\end{definition}

\begin{corollary}[Absence of Phase Transition for Brownian Potentials]
Almost surely the dual of the Ruelle operator $\mathscr{L}^{*}_{B_{t(\cdot)}}$
has one eigenprobability associated to $\lambda_{B}$.
\end{corollary}
\begin{proof}
From the Theorem \ref{prop-lambda-B-random-variable} follows that 
$\lambda_B$ is both the spectral radius and the main eigenvalue of 
$\mathscr{L}_{B_{t(\cdot)}}$. So the corollary follows from the almost certain $\gamma$-H\"older 
continuity of the mapping $\Omega\ni x\mapsto B_{t(x)}$, 
obtained in the proof of Theorem \ref{prop-lambda-B-random-variable},  
together with the item ii) of Theorem \ref{teo-rpf-holder-case}.
\end{proof}

\section{$\pmb{C(\Omega)\simeq \mathscr{D}[0,1]}$ } \label{sec-D-sim-C}

Let us denote by $\mathscr{D}$ the subset of the unit 
interval $[0,1]$ such that all of its points have exactly two 
distinct binary expansion representations, i.e.,  
\[
\mathscr{D}\equiv \{s\in[0,1]: \#t^{-1}(s)=2 \}.
\]

If $t(x)\in \mathscr{D}$ then there is a unique
$y\in\Omega\setminus\{x\}$ such that $t(x)=t(y)$.
Moreover such $x$ and $y$ 
are always comparable in the lexicographic order.
When $x$ is smaller than $y$ in this order 
we indicate this by using the notation $x\prec_{\mathrm{Lex}}y$.

We now introduce a proper subset of the Skorokhod space that 
will play a fundamental role in this section.
More details on the Skorokhod space can be found in \cite{MR1700749}.
The functions belonging to the subspace mentioned above are c\`adl\`ag on 
the unit interval but only allowed to jump
on the set $\mathscr{D}$. We use the convenient notation 
$\mathscr{D}[0,1]$ to denote this subspace and
it is characterized as follows
\[
\mathscr{D}[0,1]
\equiv
\left\{
F:[0,1]\to\mathbb{R}:\ 
\begin{array}{c}
F\ \text{is continuous in}\ [0,1]\setminus\mathscr{D} \ \text{and}\ \forall t\in\mathscr{D}
\\
\text{the left limit}\  \lim_{s\uparrow t} F(s)\ \text{exists};
\\
\text{the right limit}\ \lim_{s\downarrow t} F(s)=F(t).
\end{array}
\right\}.
\]

We want to relate this space to $C(\Omega)$ and to this end 
we introduce the following linear operator 
$\Theta: \mathscr{D}[0,1]\to C(\Omega)$
which sends a function $F\in \mathscr{D}[0,1]$ to
a function $f\in C(\Omega)$ whose its definition is 
\begin{itemize}
\item $f(x)=F(t(x))$ if $x\notin t^{-1}(\mathscr{D})$ and
\item if $x\in t^{-1}(\mathscr{D})$ then there exists a unique 
$y\in\Omega\setminus\{x\}$ 
such that $t(x)=t(y)$. If $x\prec_{\mathrm{Lex}} y$
we put $f(x)=\lim_{s\uparrow t(x)} F(s)$ and 
$f(y) = \lim_{s\downarrow t(y)} F(s)= F(t(y))$.
\end{itemize}

The linearity of $\Theta$ is obvious, but the statement 
$\Theta(F)\in C(\Omega)$ whenever $F\in \mathscr{D}[0,1]$
requires a proof and this is the content of the next proposition.

\begin{proposition}
For every $F\in \mathscr{D}[0,1]$ we have that
$\Theta(F) \in C(\Omega)$.
\end{proposition}

\begin{proof}
Let $x\in\Omega$ be an arbitrary point and $F\in \mathscr{D}[0,1]$. 
In order to prove the continuity of $\Theta(F)$ at $x$ we 
split the analysis in two cases: whether $F$ is 
continuous or not in $t(x)$. Suppose that $F$ is continuous
in $t(x)$. In this case independently whether $t(x)$ 
belongs or not to $\mathscr{D}$
we have $\Theta(F)(x)=F(t(x))$. Given $\varepsilon>0$ follows from the 
continuity of $F$ that there is $\delta>0$ such that if $|s-t(x)|<\delta$
then $|F(s)-F(t(x))|<\varepsilon$. From the continuity of $t$ it is possible
to find $\eta>0$ such that if $d(y,x)<\eta$ then $|t(x)-t(y)|<\delta$.  
Suppose by contradiction that $\Theta(F)$ is not continuous at $x$.
So there is a sequence $y_n\in\Omega$ such that $y_n\to x$ but 
$|\Theta(F)(y_n)-\Theta(F)(x)|=|\Theta(F)(y_n)-F(t(x))|\geq \varepsilon$
for all $n\in\mathbb{N}$. Because of the previous observations for $n$
large enough $\Theta(F)(y_n)\neq F(t(y_n))$ otherwise we reach a contradiction. 
So the only possibility is 
$\Theta(F)(y_n) = \lim_{s\uparrow t(y_n)}F(s)$ for all $n\geq n_0$.
Again, we reach a contradiction since for large enough $n$ we have 
$t(y_n)$ close to $t(x)$ and therefore by taking 
$s$ close enough to $t(x)$
we have $F(s)$ close to $F(t(x))$ which is a contradiction with
$|\Theta(F)(y_n)-F(t(x))|\geq \varepsilon$.   
 
Now we assume that $F$ has a jump at $t(x)$. In this case
by the definition of $\mathscr{D}[0,1]$ there exist exactly
one point $y\in\Omega\setminus\{x\}$ such that $t(x)=t(y)$. 
We give the argument for $x\prec_{\mathrm{Lex}} y$ on the 
other case the analysis is similar. Note that
\[
\begin{array}{rl}
y=&(x_1,\ldots,x_n,1,0,0,\ldots)\\
x=&(x_1,\ldots,x_n,0,1,1,\ldots).
\end{array}
\]
From the definition we have 
\[
\Theta(F)(x) = \lim_{s\uparrow t(x)} F(s).
\]
Given $\varepsilon>0$ there is $\delta>0$ such that for all $s$
satisfying $t(x)-\delta<s\leq t(x)$ we have $|F(s)-\Theta(F)(x)|<\varepsilon$.
Note that for every $z\in\Omega$ such that $d(z,x)\leq 2^{n+1}$
we have $t(z)\leq t(x)$. By taking $z$ close enough to $x$ 
and proceed similarly to the previous case 
one can show that $\Theta(F)(z)$ is close to 
$\Theta(F)(x)$, thus proving the continuity of $\Theta(F)$ at $x$.
Since $x$ is arbitrary the proof is complete.  
\end{proof}

\begin{theorem}\label{teo-banacho-iso}
The operator $\Theta$ is a Banach isomorphism (linear isometry) between the
Banach spaces $(\mathscr{D}[0,1],\|\cdot\|_{\infty})$ 
and $(C(\Omega),\|\cdot\|_{\infty})$.
\end{theorem}
\begin{proof}
We first construct the inverse operator $\Theta^{-1}$
and then we proceed to show that $\Theta$ is an isometry. 

Let $f\in C(\Omega)$ be an arbitrary continuous function.  
If $s\in [0,1]\setminus\mathscr{D}$ 
we define $F(s)=f(t^{-1}(s))$. If $s\in\mathscr{D}$ then
$t^{-1}(s)=\{x,y\}$ with $x\neq y$. 
If $x\prec_{\mathrm{Lex}}y$ we put $F(s)=f(y)$.
We claim that $F\in \mathscr{D}[0,1]$. Indeed, 
if $s\in [0,1]\setminus\mathscr{D}$, 
then $s$ has a unique binary expansion 
$t^{-1}(s)=(x_1,\ldots,x_n,\ldots)$. 
For each $n\in\mathbb{N}$ and $j=1,\ldots,2^n$ 
we consider the following collection 
of open intervals 
\[
I_n^j=\left( \frac{j-1}{2^n},\frac{j}{2^n} \right)
\]
Note that 
$\mathscr{D}\cup\{0,1\}
= 
\cup_{n\geq 1}\cup_{j=1}^{2^n} \partial I_n^j$.
Note that if $s\in [0,1] \backslash \mathscr{D}$, then $s$ belongs to the interior
of infinitely many 
$I_n^j$'s. 
Given $\varepsilon>0$ there is  
$\delta>0$ such that for all $y\in\Omega$ 
satisfying $d(t^{-1}(s),y)<\delta$
we have $|f(t^{-1}(s))-f(y)|<\varepsilon$.
As mentioned before we can choose $n$ as large as we please
so that $s\in I_n^j$, for some $j=1,\ldots, 2^n$. 
From the definition of $I_n^j$
the first $n$ terms of the binary expansion 
of any $r\in I_n^j$ coincides with $x_1,\ldots,x_n$.
For $n$ such that $2^{-n}<\delta$ we have 
$d(t^{-1}(s),t^{-1}(r))<\delta$ and therefore 
for all $r\in I_n^j$ we have 
$|F(r)-F(s)|\leq \max_{z\in t^{-1}(r)}|f(z)-f(x)|
<\varepsilon$.

Suppose that $t(x)\in\mathscr{D}$ and $\{x,y\}=t^{-1}(x)$.
Without loss of generality we can assume that $x\prec_{\mathrm{Lex}} y$.
For $n\geq n_0$ we have $t(x)\in \partial I_n^j\cap \partial I_n^{j+1} $
for a unique value of $j\in \{1,\ldots,2^n\}$. Every element in $t^{-1}(I_n^j)$
has its first $n$ digits in the binary expansion equals to $x_1,x_2,\ldots,x_n$.
Therefore some small neighborhood of $x$ is sent by $t$ in $I_n^j\cup \{t(x)\}$. 
From the continuity of $f$ at $x$ it follows that the small neighborhood
defined above is sent in a $\varepsilon$ neighborhood of $f(x)$ which
proves the existence of the left limit of $F$ at $t(x)$. 
The right continuity of $F$ at $t(x)\in\mathscr{D}$ 
is proved in similar way.

Notice that the argument in the two previous paragraphs 
also proves that the mapping
\[
\begin{array}{rl}
C(\Omega)\ni f
\mapsto
&F
\\
&
s\mapsto 
\begin{cases}
f(t^{-1}(s)),&\text{if}\ s\in [0,1]\setminus\mathscr{D};
\\
f(y),&\text{if}\ t^{-1}(s)=\{x,y\}\ \text{and}\ x\prec_{\mathrm{Lex}}y.
\end{cases}
\end{array}
\]
is a left inverse of $\Theta$. Of course, $\Theta$ is right 
inverse of this operator so $\Theta$ is a bijection.

Remains to prove that $\Theta$ is an isometry. 
From its very definition we have 
\[
\| \Theta(F) \|_{\infty}
=
\sup_{s\in [0,1]} |\Theta(F)(s)|
\leq \|F\|_{\infty}.
\]
If the supremum above is attained at certain point 
$s\in [0,1]\setminus\mathscr{D}$, take $x \in \Omega$ such that $t(x)$, then we have  
$\|F\|_{\infty}=|F(s)|=|\Theta (F)(x)|\leq \|\Theta(F)\|_{\infty}$.
If the supremum is not attained for some $s\in\mathscr{D}$ the conclusion
is analogous, otherwise 
$\|F\|_{\infty}= \lim_{n\to\infty}\lim_{r\uparrow s_n}|F(r)|
=\lim_{n\to\infty}|\Theta(F)(x_n)|$, where $x_n$ is the smallest point, with respect to 
the lexicographic order, in $t^{-1}(s_n)$ and the proposition follows.
\end{proof}

\begin{remark}[Wiener Spaces]
Recall that the classical Wiener space on the 
closed unit interval $[0,1]$, notation $W[0,1]$, 
is the set of all real continuous functions $F:[0,1]\to \mathbb{R}$
such that $F(0)=0$.
Consider the closed subspace of $C(\Omega)$ given by 
\[
W(\Omega)\equiv
\{f\in C(\Omega): f(0^{\infty})=0,\ f(x)=f(y)\ \text{whenever}\ t(x)=t(y) \},
\] 
where the mapping $t:\Omega\to\mathbb{R}$ was previously defined.
It is easy to see that for every $F\in W[0,1]$ its composition 
with $t$ give us a mapping $F\circ t \in W(\Omega)$. 
Moreover the linear mapping $F\mapsto F\circ t$
is an isometry from $W[0,1]$ to $W(\Omega)$.
\end{remark}

\section{Stochastic Functional Equation for $\pmb{\lambda_{B}}$}\label{sec-fun-eq}

In this section we apply the results above obtained. 
The idea is to obtain a stochastic functional equation 
using the operator $\Theta$ and Proposition \ref{prop-auto-val-func-sao-rv}
to semi-explicit representation of the main eigenvalue $\lambda_{B}$
of the random Ruelle operator $\mathscr{L}_{B_{t(\cdot)}}$. 

Since $h_{B}$ is  a continuous function almost surely,
we can associate to it a {\it c\`adl\`ag} process $\{X_s: 0\leq s\leq 1\}$
so that $\Theta(X_{(\cdot)})=h_{B}$. 
By simple algebraic manipulation we can see that 
$t(0x)=t(x)/2$ and $t(1x)=1/2+t(x)/2$.
By applying the inverse of the 
operator $\Theta$ on both sides of the 
above equation we get the following 
{\it stochastic functional equation} 
\begin{align*}
\exp(B_{\frac{t}{2}})X_{\frac{t}{2}}
+
\exp(B_{\frac{1}{2}+\frac{t}{2}})X_{\frac{1}{2}+\frac{t}{2}}
=
\lambda_{B}X_t.
\end{align*}

Now we have to play with the properties of the Brownian motion
and the Ruelle operator to obtain the law of the main eigenvalue.
Recalling that $B_0=0$ almost surely, so we get from the stochastic 
functional equation that 
\[
X_{0}
+
\exp(B_{\frac{1}{2}})X_{\frac{1}{2}}
=
\lambda_{B}X_0.
\]
Therefore we get by isolating the eigenvalue 
(recall that $X_s>0$ for each $s\in[0,1]$) the following equality a.s.
\begin{align}\label{eq-expressao-lambdaB}
\lambda_{B} 
=
1+\exp(B_{\frac{1}{2}})\frac{X_{ \frac{1}{2} } } {X_0} 
\end{align}

By using again the operator $\Theta$ one can see that 
$X_0=h_{B}(0,0,\ldots)$ and $X_{1/2}=h_B(1,0,0,\ldots)$,
recall  that 
$(0,1,1,\ldots)\prec_{\mathrm{Lex}}(1,0,0,\ldots)$.
From the Theorem \ref{teo-rpf-holder-case} 
we have almost certain
\[
\frac{X_{ \frac{1}{2} } } {X_0}
=
\frac{h_B(1,0,0,\ldots)}{h_{B}(0,0,\ldots)}
=
\lim_{n\to\infty}
\frac{\mathscr{L}^{n}_{B_{t(\cdot)} }(1)(1,0,0,\ldots)}
{\mathscr{L}^{n}_{B_{t(\cdot)} }(1)(0,0,0,\ldots)}.
\]
From where we conclude that 
\[
\lambda_{B} 
=
1+\exp(B_{\frac{1}{2}})
\lim_{n\to\infty}
\frac{\mathscr{L}^{n}_{B_{t(\cdot)} }(1)(1,0,0,\ldots)}
{\mathscr{L}^{n}_{B_{t(\cdot)} }(1)(0,0,0,\ldots)}.
\]

\section{Bounds on the Expected Value of $\pmb{\lambda_{B}}$} \label{sec-bounds-lambda}

In the deterministic case, when the potential is assumed to 
have H\"older, Walters or Bowen regularity, it is easy to obtain 
lower and upper bounds of the main eigenvalue of
the Ruelle operator by using the supremum norm of the potential.
We recall that when the potential have such regularity 
the main eigenvalue do exists and is actually equals
to the spectral radius. 

Since here we are considering random potentials 
lower and upper bounds for the expected values, 
with respect to the Wiener measure, 
of the spectral radius is a natural question to ask.

Opposed to the deterministic case 
the supremum norm of $B_{t(\cdot)}$ by itself does not help 
in finding bounds to the spectral radius. 
The reason is $\sup\{|B_{t(x)}|:x\in\Omega \}$ is an unbounded random 
variable. 
The representation of $\lambda_{B}$ obtained in the last
section is not suitable to take expectations 
and very hard to bound due its complex combinatorial nature.
Therefore a different approach
is needed to bound the expected value of main eigenvalue 
and the idea is based on the reflection principle of the Brownian motion.

\begin{theorem}\label{prop-estimativa-valor-esperado-lambda}
Let $\lambda_B$ be the random variable defining the spectral radius of the 
random Ruelle operator $\mathscr{L}_{B_{t(\cdot)}}$.
Then $\lambda_B$ has finite first moment, with respect to 
the Wiener measure, and moreover 
\[
\exp(1/2)
\leq 
\mathbb{E}[\lambda_{B}] 
\leq 
4 \exp{(1/2)}.
\]
\end{theorem}

\begin{proof}
We begin with the lower bound. To lighten  
the notation we use the symbol $1^{\infty}$ to denote the constant 
sequence $(1,1,1,\ldots)$.  
By the definition of the 
supremum we have 
\begin{align*}
\sup_{\varphi:\|\varphi\|_{\infty}=1}
\| \mathscr{L}_{B_{t(\cdot)}}^{n}(\varphi) \|_{\infty}
&=
\sup_{\varphi:\|\varphi\|_{\infty}=1}
\ \sup_{x\in\Omega}
|\mathscr{L}_{B_{t(\cdot)}}^{n}(\varphi)(x)|
\\
&\hspace*{-1.6cm}=
\sup_{\varphi:\|\varphi\|_{\infty}=1}
\ \sup_{x\in\Omega}
| \sum_{a_1,\ldots,a_n} 
\exp(
	\sum_{j=0}^{n-1} 
	B_{t(\sigma^j(a_1\ldots a_nx))} 
	)
\varphi(\sigma^j(a_1\ldots a_nx))|
\\
&\hspace*{-1.6cm}\geq 
\exp
	\sum_{j=0}^{n-1} 
	B_{t(1^{\infty})} 
\\
&\hspace*{-1.6cm}\geq 
\exp(n B_{1}).
\end{align*}
By taking the $n$-th root, the limit when $n\to\infty$ and the 
expectation in the above inequality 
we get that 
\[
\mathbb{E}[\lambda_{B}]
\geq 
\mathbb{E}[\exp(B_1)]
=
\exp(1/2).
\]

To get the upper bound we will take advantage of the 
reflection principle of the Wiener process. 
We first observe that 
\begin{align*}
\sup_{\varphi:\|\varphi\|_{\infty}=1}
\| \mathscr{L}_{B_{t(\cdot)}}^{n}(\varphi) \|_{\infty}
&\leq
\ \sup_{x\in\Omega}
|\mathscr{L}_{B_{t(\cdot)}}^{n}(1)(x)|
\\
&=
\sup_{x\in\Omega}
\sum_{a_1,\ldots,a_n} 
\exp
	\sum_{j=0}^{n-1} 
	B_{t(\sigma^j(a_1\ldots a_nx))} 
\\
&\leq
\sum_{a_1,\ldots,a_n} 
\sup_{x\in\Omega}\ 
\exp
	\sum_{j=0}^{n-1} 
	B_{t(\sigma^j(a_1\ldots a_nx))} 
\\
&\leq 
2^n	\exp(n \max\{B_{t}: 0\leq t\leq 1\}).
\end{align*}
By taking the $n$-th root and the limit
when $n\to\infty$ in the last inequality
we get 
$\lambda_{B} \leq 2\exp(M_1)$,
where $M_1\equiv\max\{B_{t}: 0\leq t\leq 1\}$.
From the reflection principle of the Wierner process 
it follows that 
\begin{equation*}
\mathbb{E}[\lambda_{B}] 
\leq
2\mathbb{E}[\exp(M_1)]
\le
4\cdot \mathbb{E} [\exp{B_1}]
=
4 \exp(1/2).
\qedhere
\end{equation*}
\end{proof}

\section{Existence and Finiteness of the Quenched Pressure}\label{sec-quenched-pressure}
By using the $\gamma$-H\"older regularity of the 
Brownian potential $x\mapsto B_{t(x)}$ and a classical 
result of Thermodynamic formalism,  we have almost certain 
\[
\log \lambda_{B}  
=
P( B_{t(\cdot) } )
=
\sup_{\mu\in \mathscr{P}_{\sigma}(\Omega)}
\left\{h(\mu)+\int_{\Omega} B_{t(\cdot)}\, d\mu\right\},
\]
where $P(\cdot)$ is the topological pressure,
$\mathscr{P}_{\sigma}(\Omega)$ is the set of all Borel
probability measures that are shift invariant and $h$
is the Kolmogorov-Sinai entropy.  
From \eqref{eq-expressao-lambdaB} we know that
$\lambda_{B} 
=
1+\exp(B_{1/2})(X_{1/2}/X_0),
$
and $X_{1/2}/X_0\geq 0$.
Therefore we can conclude that 
$\log \lambda_{B}$ is non-negative random variable and 
its expected value with respect to the Wiener measure is well 
defined, at least as an extended real number.  

Finally, by using the elementary inequality 
$\log \lambda_{B}\leq \lambda_B$
and the Theorem \ref{prop-estimativa-valor-esperado-lambda} 
we have that $\mathbb{E}[\log \lambda_{B}]<+\infty$,
thus proving the finiteness of the quenched pressure.
Using once more the $\gamma$-H\"older regularity of the 
Brownian potential we have almost certain that 
\[
\log\lambda_{B}
=
\lim_{n\to\infty} 
\frac{1}{n} \log \mathscr{L}^n_{B_{t(\cdot)}}(1)(\sigma^n (x)),
\qquad \forall x\in\Omega
\]
and therefore 
\[
P^{\mathrm{quenched}}(B_{t(\cdot)})
\equiv
\mathbb{E}[\log\lambda_{B}]
=
\mathbb{E}\left[
\lim_{n\to\infty} 
\frac{1}{n} 
\log \mathscr{L}^n_{B_{t(\cdot)}}(1)(\sigma^n (x))
\right].
\]

Finally, by the Theorem \ref{prop-estimativa-valor-esperado-lambda}  
and the Jensen Inequality we get 
\[
    0 \le P^{\mathrm{quenched}}(B_{t(\cdot)}) = \mathbb{E}[\log\lambda_{B}]
    \le \log \mathbb{E}[\lambda_{B}] \le  \log 4 + \frac{1}{2}.
\]

\section{Concluding Remarks} \label{sec-remarks}

As mentioned, all the results presented in this paper can extended to the space
$\Omega = \{0, \dots, m-1\}^\mathbb{N}$. 
For such state spaces the map $t:\Omega\to [0,1]$ 
is given by $(x_1, x_2 \dots ) \mapsto \sum_{n\geq 1}m^{-n}x_n$.
Analogous arguments can be used to prove the Proposition
\ref{prop-lambda-B-random-variable} and Proposition \ref{prop-auto-val-func-sao-rv}.
In the Section \ref{sec-D-sim-C}, the arguments would be completely analogous, 
only changing appropriately the set $\mathscr{D}$. 

In the Section \ref{sec-fun-eq}, we would have a subtle change in 
the stochastic functional equation.
Now we have for every $a\in \{0,\dots, m-1\}$ and $x \in \Omega$,
$t(ax)= a/m + t(x)/m$, so a slightly different stochastic 
functional equation shows up
$
\exp(B_{t/m})X_{t/m}
+
\exp(B_{1/m+t/m})X_{1/m+t/m}
=
\lambda_{B}X_t.
$
Thus, by taking $t=0$, we recover the representation of the main
eigenvalue 
\begin{align*}
\lambda_{B} 
&=
1+\exp(B_{\frac{1}{m}})
\lim_{n\to\infty}
\frac{\mathscr{L}^{n}_{B_{t(\cdot)} }(1)(1,0,0,\ldots)}
{\mathscr{L}^{n}_{B_{t(\cdot)} }(1)(0,0,0,\ldots)}.
\end{align*}
The technique employed in Theorem \ref{prop-estimativa-valor-esperado-lambda}
provided the following bounds
\[
\exp(1/2)
\leq 
\mathbb{E}[\lambda_{B}] 
\leq 
2m \exp(1/2).
\]
Therefore, by using the same arguments, one can also prove the existence and finiteness
of the Quenched Pressure on this case.

Another natural question is whether we could study another random potential defined by a  continuous time stochastic process other than a
Brownian Motion. Note that the only property of the Brownian 
Motion used to guarantee the existence of the main eigenvalue is the almost 
certain H\"older continuity of the paths (in fact this is required only in 
$[0,1] \backslash \mathscr{D}$ ), which is not an exclusive property of 
Brownian Motion. In the functional stochastic equation, we only used that $B_0 = 0$ 
almost surely, but for the general case all that is required is to add a constant to
the potential. Here we used the reflection principle and the stationarity of the 
increments to bound $\mathbb{E}[\lambda_{B}]$ which also follow for some other processes. Nevertheless
if one has another technique to guarantee its finiteness, the general sequence described in the article still holds and the existence and finiteness of the Quenched Pressure is still guaranteed.
Therefore, the natural candidate to extend the results is Fractional Brownian 
motion, which indeed leads to very analogous results. 

The obstacle one faces to obtain tighter lower 
bounds for $\mathbb{E}[\lambda_{B}]$ 
is of combinatorial nature and the complicated sums of dependents 
log-normal random variables that appear in the calculations and have to be controlled.  
Some informal computations suggest that the annealed pressure 
is also well defined for Brownian type potential and possibly given by 
$\log (\mathrm{ess\, sup} \lambda_{B})$, but unfortunately 
the techniques developed here seem not so helpful in attacking this problem.  
\vspace*{-0.5cm}
\section*{Acknowledgments}
\vspace*{-0.2cm}
The authors gratefully acknowledge Eduardo A. Silva, Hannah Goozee 
and the professors Mike Todd and Yuri Dumaresq Sobral
for many helpful remarks and suggestions on the early versions
of this manuscript. We also thank to 
professor Yuri Kifer for several useful references.
Leandro Chiarini and 
Leandro Cioletti acknowledge CNPq for financial support.

\vspace*{-0.5cm}
\bibliographystyle{alpha}
\bibliography{referencias}

\begin{flushright}
\bigskip
{\sc 
Leandro Chiarini\\
Leandro Cioletti
\\[0.2cm]
Departamento de Matem\'atica - UnB
\\
Campus Universit\'ario Darcy Ribeiro - Asa Norte\\
70910-900  Bras\'ilia - DF - Brazil.}
\\
{\it chiarini@mat.unb.br}\\
{\it cioletti@mat.unb.br}
\end{flushright}

\end{document}